\documentclass{amsart}

\usepackage{graphicx}
\usepackage{enumerate}
\usepackage{amsmath,amssymb}
\usepackage{mathrsfs}

\newtheorem{maintheorem}{Theorem}

\newtheorem{theorem}{Theorem}[section]
\newtheorem{lemma}[theorem]{Lemma}

\newtheorem{corollary}[theorem]{Corollary}

\theoremstyle{definition}
\newtheorem{definition}[theorem]{Definition}
\newtheorem{remark}[theorem]{Remark}

\newcommand{\Diff}{\operatorname{Diff}}

\newcommand{\dist}{\operatorname{dist}}

\begin{document}

\title[Pluripotency and Anosov Dynamics
]
{
Pluripotency and Anosov Dynamics
}

\author{Shin Kiriki}
\address[Shin Kiriki]
{Department of Mathematics, Tokai University, 4-1-1 Kitakaname, Hiratsuka, Kanagawa, 259-1292, Japan}
\email{kiriki@tokai.ac.jp}

\author{Xiaolong Li}
\address[Xiaolong Li]
{School of Mathematics and Statistics,
Huazhong University of Science and Technology, Luoyu Road 1037, Wuhan, 430074, China}
\email{lixl@hust.edu.cn}
\thanks{X.L. was supported by NSFC 11701199 and 12331005.}

\author{Kodai Yamamoto}
\address[Kodai Yamamoto]
{Joint graduate school of mathematics for innovation, Kyushu University, 744 Motooka, Nishi-ku, Fukuoka, 819-0395, Japan}
\email{yamamoto.kodai.508@s.kyushu-u.ac.jp}
\thanks{K.Y. was supported by WISE program (MEXT) at Kyushu University.}


\date{\today}

\begin{abstract}

We study two notions of pluripotency for dynamical systems. Strong pluripotency refers to the possibility of reprogramming, by arbitrarily small perturbations, the forward orbits of points in a set of positive Riemannian volume so that they track a prescribed orbit in a uniform time-averaged sense. In contrast, pluripotency, also called ``statistical pluripotency" in this paper, requires only asymptotic agreement of empirical measures. Strong pluripotency implies statistical pluripotency, but the converse need not hold.

For $C^r$ transitive Anosov diffeomorphisms, $r\ge2$, we prove that strong pluripotency fails for every nonempty reference set.  On the other hand, statistical pluripotency holds for a nonempty reference set if and only if it is contained in the basin of the unique SRB measure. This basin is invariant and has full Riemannian volume, but contains no periodic points. These results reveal a sharp distinction between the two notions even within uniformly hyperbolic dynamics.
\end{abstract}

\maketitle

\section{Introduction}

The notions of pluripotency and strong pluripotency share a common feature:
under arbitrarily small perturbations, the empirical measures of all points in a set of positive Riemannian
volume asymptotically agree with those of a prescribed reference orbit.
Let $M$ be a compact Riemannian manifold without boundary,
let $f:M\to M$ be a map, and let $x\in M$.
For $n\in\mathbb N$, the \emph{empirical measure} of $x$ is defined by
\[
\delta^n_{x,f}
=
\frac{1}{n}\sum_{i=0}^{n-1}\delta_{f^i(x)},
\]
where $\delta_z$ denotes the Dirac measure supported at $z\in M$.

The prescribed asymptotic dynamics includes not only
the case in which these empirical measures converge
to an invariant measure, but also historic behavior,
as considered by Ruelle~\cite{R01} and Takens~\cite{T08},
in which they fail to converge.

The term pluripotency is borrowed from stem-cell biology.
More precisely, it is inspired by Yamanaka's Nobel Prize-winning work
on induced pluripotent stem cells: by suitably reprogramming
differentiated cells, one can endow them with the ability to develop
into many different cell types.

Before giving a dynamical formulation of pluripotency, let us briefly
recall its origins. Independently of the questions raised by Ruelle and
Takens, Colli and Vargas \cite{CV01} presented a two-dimensional model
capable of stochastically realizing several types of asymptotic behavior,
such as historic behavior and nontrivial Dirac physical measures,
according to prescribed binary codes. Their model was built on Newhouse's
two-dimensional example of a thick affine horseshoe with persistent
tangencies \cite{N70}, using infinitely many perturbations of homoclinic
tangencies associated with the horseshoe. Although their work did not
address whether these properties were persistent, the basic mechanism of
pluripotency was already present in this pioneering construction.
Later, Kiriki et al.~\cite{KLNSV25} showed that an analogous pluripotent
mechanism persists in a $C^r$-neighborhood, $r\geq2$, of the
Colli-Vargas model.

The mechanism used there is based on a basic set whose dynamics is coded
by two symbols, namely a horseshoe, together with a homoclinic tangency.
Inside the gaps associated with the homoclinic tangency, one constructs
wandering domains of positive Lebesgue measure whose orbits reprogram
arbitrary dynamics of the horseshoe.
Horseshoes also play a fundamental role in the
approximation of hyperbolic invariant measures,
as illustrated by Katok-type constructions;
see, for instance, Gelfert~\cite{G16}.

This naturally raises the question of whether pluripotency can go
beyond horseshoes and reprogram, on sets of positive Lebesgue measure,
the dynamics of more general Axiom~A systems. As a first global setting
for this question, in this paper we focus on Anosov dynamics.

We now give a precise formulation of pluripotency. We use the
\emph{first Wasserstein metric}, defined for Borel probability measures
$\mu$ and $\nu$ on $M$ by
\[
d_W(\mu,\nu)
=
\sup_{\varphi\in\mathcal L}
\left|
\int_M\varphi\,d\mu-\int_M\varphi\,d\nu
\right|,
\]
where $\mathcal L$ denotes the set of Lipschitz functions
$\varphi:M\to\mathbb R$ with Lipschitz constant at most $1$.

Let $r\ge1$, let $\Lambda$ be a locally maximal uniformly hyperbolic
compact transitive invariant set for $f\in\Diff^r(M)$, and let
$\Lambda'\subset\Lambda$ be an arbitrary nonempty subset. For every $g$ sufficiently
$C^1$-close to $f$, let $\Lambda_g$ denote the hyperbolic continuation
of $\Lambda$, and let
$
H_g:\Lambda\to\Lambda_g
$
be the conjugacy satisfying
$
H_g\circ f=g\circ H_g$.
For $x\in\Lambda$, we write
\[
x_g:=H_g(x)
\]
for its continuation.

We consider the following two notions of pluripotency.

\begin{definition}[Pluripotency]\label{def:Cr-pluripotency}
We say that $f$ is \emph{$C^r$-pluripotent for $\Lambda'$} if, for every
$x\in\Lambda'$ and every $C^r$-neighborhood $\mathcal U$ of $f$, there
exist $g\in\mathcal U$ and a measurable set $D_g\subset M$ of positive
Riemannian volume such that
\[
\lim_{n\to\infty}
d_W(\delta_{y,g}^n,\delta_{x_g,g}^n)=0
\qquad
\text{for every }y\in D_g.
\]
\end{definition}
Since this notion is formulated in terms of the asymptotic
agreement of empirical measures, we also refer to it as
\emph{statistical pluripotency} in this paper.

We next introduce a stronger notion of pluripotency.

\begin{definition}[Strong pluripotency]\label{def:Cr-strong-pluripotency}
We say that $f$ is \emph{$C^r$-strongly pluripotent for $\Lambda'$} if,
for every $x\in\Lambda'$ and every $C^r$-neighborhood $\mathcal U$ of
$f$, there exist $g\in\mathcal U$ and a measurable set $D_g\subset M$
of positive Riemannian volume such that
\[
\lim_{n\to\infty}
\frac1n\sum_{i=0}^{n-1}
\sup_{y\in D_g}
\dist\bigl(g^i(y),g^i(x_g)\bigr)
=0.
\]
\end{definition}
By the definition of $d_W$, $C^r$-strong pluripotency implies
$C^r$-pluripotency, whereas the converse does not hold in general;
see \cite[Theorem~1.8]{KNS}.
Strong pluripotency requires uniform time-averaged closeness
to the reference orbit, whereas statistical pluripotency requires
only asymptotic agreement of empirical measures.
In both notions, the diversity of asymptotic behaviors that
can be reproduced depends on the choice of the reference set
$\Lambda'$.

Our first main result gives a complete obstruction to strong pluripotency
in the Anosov setting.

\begin{maintheorem}\label{thm:absence-strong-pluripotency}
Let $r\ge2$, and let $f:M\to M$ be a $C^r$ transitive Anosov
diffeomorphism on a compact connected Riemannian manifold $M$.
Then $f$ is not $C^r$-strongly pluripotent for any nonempty subset
$\Lambda'\subset M$.
\end{maintheorem}

We next turn to statistical pluripotency.
The following theorem gives a complete characterization of the nonempty subsets for which this property holds.
For a transitive Anosov diffeomorphism, these are exactly the nonempty subsets of the basin of its unique SRB measure.

\begin{maintheorem}\label{thm:characterization-statistical-pluripotency}
Let $r\ge2$, and let $f:M\to M$ be a $C^r$ transitive
Anosov diffeomorphism on a compact connected Riemannian
manifold $M$.
Let $\mu_f$ be its unique SRB measure, and let
$B_f(\mu_f)$ denote its basin.
Then, for every nonempty subset $\Lambda'\subset M$,
\[
f\text{ is }C^r\text{-pluripotent for }\Lambda'
\quad\Longleftrightarrow\quad
\Lambda'\subset B_f(\mu_f).
\]
Moreover, $B_f(\mu_f)$ is an $f$-invariant measurable
set of full Riemannian volume.
\end{maintheorem}

As we show in Section~\ref{s.pluripotency}, the basin of the unique SRB measure contains no periodic point.
Since a transitive Anosov diffeomorphism has periodic points, this basin is a proper subset of $M$.
Taking $\Lambda'=M$ in Theorem~\ref{thm:characterization-statistical-pluripotency} gives the following immediate consequence.

\begin{corollary}\label{cor:absence-pluripotency}
  Let $r\geq2$.
  No $C^r$ transitive Anosov diffeomorphism on a compact connected Riemannian manifold $M$ is $C^r$-pluripotent for $M$.
\end{corollary}

\begin{remark}[Robustness of the main results]
  For each fixed $r\ge2$, the conclusions of Theorem~\ref{thm:absence-strong-pluripotency}, Theorem~\ref{thm:characterization-statistical-pluripotency} and Corollary~\ref{cor:absence-pluripotency} are $C^r$-robust.
  Indeed, by the $C^1$ openness and structural stability of Anosov diffeomorphisms, there exists a $C^1$-neighborhood $\mathcal V$ of $f$ such that every $g\in\mathcal V$ is Anosov and topologically conjugate to $f$. Hence there exists a $C^r$-neighborhood $\mathcal U$ of $f$ such that $ \mathcal U\subset \mathcal V\cap\Diff^r(M).$
  Every $g\in\mathcal U$ is therefore a transitive Anosov diffeomorphism.
  Consequently, every such $g$ is neither $C^r$-strongly pluripotent for any nonempty subset of $M$ nor $C^r$-pluripotent for $M$.

  Nevertheless, for every $g\in\mathcal U$ and every nonempty subset $\Lambda'\subset M$, $g$ is $C^r$-pluripotent for $\Lambda'$ if and only if $\Lambda'\subset B_g(\mu_g)$, where $\mu_g$ is the unique SRB measure of $g$.
  The basin $B_g(\mu_g)$ is a $g$-invariant measurable set of full Riemannian volume.
\end{remark}

The pluripotency in Theorem~\ref{thm:characterization-statistical-pluripotency} arises from the convergence of empirical measures to a common SRB measure, rather than from the flexibility to reproduce diverse asymptotic behaviors.

The assumption $r\ge2$ is made for convenience in applying the
classical SRB theory. In view of the $C^{1+\alpha}$ results of
Bressaud and Liverani~\cite{BL02}, the regularity assumption $r\ge2$ is probably not sharp.
For related developments of SRB and physical measure theory beyond
the uniformly hyperbolic setting, see, for instance,
Mi, Cao, and Yang~\cite{MCY17}.

The remainder of the paper is organized as follows.
Section~\ref{s.absence} contains the proof of
Theorem~\ref{thm:absence-strong-pluripotency}, and
Section~\ref{s.pluripotency} contains the proof of
Theorem~\ref{thm:characterization-statistical-pluripotency}.

\section{Absence of strong pluripotency}\label{s.absence}

We use the following standard consequence of Bowen's theory of
equilibrium states for $C^2$ Anosov diffeomorphisms.

Throughout this section, $\operatorname{vol}_M$ denotes the Riemannian
volume on $M$. We equip $M\times M$ with the product Riemannian metric
and denote its Riemannian volume by $\operatorname{vol}_{M\times M}$.
Thus
\[
\operatorname{vol}_{M\times M}
=
\operatorname{vol}_M\times\operatorname{vol}_M.
\]

For a continuous map $F:M\to M$ and an $F$-invariant probability
measure $\mu$, we denote by
\[
B_F(\mu):=
\left\{
x\in M:
\lim_{n \to \infty}\frac1n\sum_{i=0}^{n-1}\varphi(F^i(x))
=
\int_M\varphi\,d\mu
\text{ for every }\varphi\in C^0(M)
\right\}
\]
the basin of $\mu$.

The following standard consequences of Bowen's theory will be useful, including the product formulation needed below.

\begin{lemma}\label{lem:SRB-basin}
Let $g:M\to M$ be a $C^2$ transitive Anosov diffeomorphism on a compact
connected Riemannian manifold $M$.
Then $g$ admits a unique SRB measure $\mu_g$ satisfying
\[
\operatorname{vol}_M(B_g(\mu_g))
=
\operatorname{vol}_M(M),
\qquad
\operatorname{supp}\mu_g=M.
\]

Moreover, if
\[
G:=g\times g:M\times M\to M\times M,
\]
then $G$ is a topologically mixing Anosov diffeomorphism,
$\mu_g\times\mu_g$ is its SRB measure, and
\[
\operatorname{vol}_{M\times M}
\bigl(B_G(\mu_g\times\mu_g)\bigr)
=
\operatorname{vol}_{M\times M}(M\times M).
\]
\end{lemma}

\begin{proof}
Let $$\phi_g^{(u)}(x):=-\log
|\det(Dg(x)|_{E^u_{g,x}})|$$
denote the unstable Jacobian potential of $g$,
where $E^u_{g,x}$ is the unstable subspace at $x$.
By Bowen~\cite[Theorem~4.1]{Bow08},
$\phi_g^{(u)}$ admits a unique equilibrium state
$\widetilde{\mu}_g$.

Since $M$ is an attractor for $g$, Bowen
\cite[Proposition~4.8]{Bow08} gives
\[
P_g(\phi_g^{(u)})=0,
\]
where $P_g(\phi_g^{(u)})$ denotes the topological pressure
of $\phi_g^{(u)}$ with respect to $g$.

Since $\widetilde{\mu}_g$ is the equilibrium state for
$\phi_g^{(u)}$, it follows that
\[
h_{\widetilde{\mu}_g}(g)
+
\int_M \phi_g^{(u)}\,d\widetilde{\mu}_g
=
0,
\]
where $h_{\widetilde{\mu}_g}(g)$ denotes the Kolmogorov--Sinai entropy.
Equivalently, by the definition of the unstable Jacobian potential,
\[
h_{\widetilde{\mu}_g}(g)
=
\int_M
\log\left|
\det\bigl(Dg_x|_{E^u_{g,x}}\bigr)
\right|
\,d\widetilde{\mu}_g(x).
\]
Thus, by the characterization of SRB measures via Pesin's
entropy formula in \cite[Theorem~2]{Y2002},
$\widetilde{\mu}_g$ is an SRB measure of $g$.
Since $M$ is a transitive Axiom A attractor,
this SRB
measure is unique by \cite[Theorem~1]{Y2002}.
Hence
\[
\widetilde{\mu}_g=\mu_g.
\]

By Bowen \cite[Theorem~4.12]{Bow08},
\[
\operatorname{vol}_M\bigl(B_g(\mu_g)\bigr)
=
\operatorname{vol}_M(M).
\]
Moreover,
\[
\operatorname{supp}\mu_g=M.
\]
Indeed, by the Gibbs estimate for the equilibrium state,
Bowen \cite[Proposition~4.4]{Bow08}, every sufficiently small
Bowen ball $B_n(x,\varepsilon)$ has positive $\mu_g$-measure.
Here, the Bowen ball $B_n(x,\varepsilon)$ is defined by
\[
B_n(x,\varepsilon)
:=\{y\in M:\dist(g^i(y),g^i(x))<\varepsilon
\text{ for }0\le i<n\},
\]
and $B(x,\varepsilon)$ is the metric ball.
Since
\[
B_n(x,\varepsilon)\subset B(x,\varepsilon),
\]
we have $\mu_g(B(x,\varepsilon))>0$. Since $\varepsilon>0$ can be
chosen arbitrarily small, every neighborhood of $x$ has positive
$\mu_g$-measure.
 Hence
$x\in\operatorname{supp}\mu_g$ for every $x\in M$, and therefore
$\operatorname{supp}\mu_g=M$.

We shall apply Bowen's theorem also to $G=g\times g$. For this purpose,
it remains to verify that $G$ is a transitive Anosov diffeomorphism and
that its SRB measure is $\mu_g\times\mu_g$.
Since $M$ is connected and $g$ is transitive Anosov, $g$ is topologically
mixing. Hence $G=g\times g$ is also topologically mixing and Anosov.

Let $\phi_G^{(u)}$ denote the unstable Jacobian potential of $G$; thus
\[
\phi_G^{(u)}(x,y)
=
-\log \left|
\det\bigl(DG_{(x,y)}|_{E^u_{G,(x,y)}}\bigr)
\right|.
\]
Since
$
E^u_{G,(x,y)}
=
E^u_{g,x}\oplus E^u_{g,y},
$
we have
$
DG_{(x,y)}|_{E^u_{G,(x,y)}}
=
Dg_x|_{E^u_{g,x}}
\oplus
Dg_y|_{E^u_{g,y}},
$
and hence
\[
\left|
\det\bigl(DG_{(x,y)}|_{E^u_{G,(x,y)}}\bigr)
\right|
=
\left|
\det\bigl(Dg_x|_{E^u_{g,x}}\bigr)
\right|
\left|
\det\bigl(Dg_y|_{E^u_{g,y}}\bigr)
\right|.
\]
Therefore,
\[
\phi_G^{(u)}(x,y)
=
\phi_g^{(u)}(x)+\phi_g^{(u)}(y).
\]

Moreover, the product formula for measure-theoretic entropy gives
\[
h_{\mu_g\times\mu_g}(G)
=
2h_{\mu_g}(g).
\]
Together with the preceding identity for $\phi_G^{(u)}$, this yields
\[
\begin{aligned}
h_{\mu_g\times\mu_g}(G)
+
\int_{M\times M}\phi_G^{(u)}\,d(\mu_g\times\mu_g)
&=
2\left(
h_{\mu_g}(g)
+
\int_M\phi_g^{(u)}\,d\mu_g
\right)\\
&=0.
\end{aligned}
\]
Since $M\times M$ is an attractor for $G$, Bowen
\cite[Proposition~4.8]{Bow08} gives
\[
P_G(\phi_G^{(u)})=0.
\]
Hence $\mu_g\times\mu_g$ is the unique equilibrium state for
$\phi_G^{(u)}$. By the same characterization as above,
\cite[Theorem~1]{Y2002}, it is the SRB measure of $G$.

Applying Bowen \cite[Theorem~4.12]{Bow08} to $G$, we obtain
\[
\operatorname{vol}_{M\times M}
\bigl(B_G(\mu_g\times\mu_g)\bigr)
=
\operatorname{vol}_{M\times M}(M\times M).
\]
This completes the proof.
\end{proof}

To prove Theorem \ref{thm:absence-strong-pluripotency},
we first establish the following consequence of
Lemma~\ref{lem:SRB-basin}.

\begin{lemma}\label{lem:positive-mean-separation}
Let $g:M\to M$ be a $C^2$ transitive Anosov diffeomorphism on a compact
connected Riemannian manifold $M$, and let $D\subset M$ be a measurable
set of positive Riemannian volume. Then there exist $y,z\in D$ such that
\[
\lim_{n\to\infty}
\frac1n\sum_{i=0}^{n-1}
\dist\bigl(g^i(y),g^i(z)\bigr)
=
c_g>0,
\]
where
\[
c_g:=
\int_{M\times M}\dist(a,b)\,
d(\mu_g\times\mu_g)(a,b).
\]
\end{lemma}

\begin{proof}
By Lemma~\ref{lem:SRB-basin}, the basin
$B_G(\mu_g\times\mu_g)$ has full Riemannian volume in $M\times M$.
Since $D$ has positive Riemannian volume, $D\times D$ has positive
Riemannian volume in $M\times M$. Hence we can choose
\[
(y,z)\in (D\times D)\cap B_G(\mu_g\times\mu_g).
\]
Applying the definition of the basin to the continuous observable
\[
\varphi(a,b):=\dist(a,b),
\]
and using $G=g\times g$, we obtain
\begin{equation*}
  \lim_{n \to \infty} \frac1n \sum_{i=0}^{n-1} \dist\bigl(g^i(y),g^i(z)\bigr)
  = \int_{M\times M}\dist(a,b)\,d(\mu_g\times\mu_g)(a,b).
\end{equation*}
The right-hand side equals $c_g$.

It remains to show that $c_g>0$.
By
Lemma~\ref{lem:SRB-basin},
\[
\operatorname{supp}\mu_g=M.
\]
Choose distinct points $p,q\in M$. Then there exist nonempty open
neighborhoods $U$ of $p$ and $V$ of $q$, and $\delta>0$, such that
\[
\dist(a,b)\ge\delta
\qquad
\text{for all }a\in U,\ b\in V.
\]
Since $\mu_g$ has full support,
\[
\mu_g(U)>0
\qquad\text{and}\qquad
\mu_g(V)>0.
\]
Consequently,
\[
c_g
\geq
\int_{U\times V}\dist(a,b)\,
d(\mu_g\times\mu_g)(a,b)
\geq
\delta\,\mu_g(U)\mu_g(V)
>0.
\]
This completes the proof.
\end{proof}

\begin{proof}[Proof of Theorem \ref{thm:absence-strong-pluripotency}]
Suppose, to the contrary, that $f$ is $C^r$-strongly pluripotent for some nonempty subset $\Lambda'\subset M$.
Fix $x\in\Lambda'$.

By structural stability, we may choose a $C^r$-neighborhood $\mathcal U$ of $f$ such that every $g\in\mathcal U$ is a transitive Anosov diffeomorphism.

By the $C^r$-strong pluripotency of $f$, there exist $g\in\mathcal U$ and a measurable set $D_g\subset M$ of positive Riemannian volume such that
\[
\lim_{n\to\infty}
\frac1n\sum_{i=0}^{n-1}
\sup_{w\in D_g}
\dist\bigl(g^i(w),g^i(x_g)\bigr)
=0.
\]

Applying Lemma~\ref{lem:positive-mean-separation} to $g$ and $D_g$, we can choose $y,z\in D_g$ such that
\[
\lim_{n\to\infty}
\frac1n\sum_{i=0}^{n-1}
\dist\bigl(g^i(y),g^i(z)\bigr)
=
c_g>0.
\]

On the other hand, for every $i\geq0$, the triangle inequality gives
\[
\begin{aligned}
\dist\bigl(g^i(y),g^i(z)\bigr)
&\leq
\dist\bigl(g^i(y),g^i(x_g)\bigr)
+
\dist\bigl(g^i(z),g^i(x_g)\bigr)\\
&\leq
2\sup_{w\in D_g}
\dist\bigl(g^i(w),g^i(x_g)\bigr).
\end{aligned}
\]
Averaging over $i=0,\ldots,n-1$, we obtain
\[
\frac1n\sum_{i=0}^{n-1}
\dist\bigl(g^i(y),g^i(z)\bigr)
\le
2\frac1n\sum_{i=0}^{n-1}
\sup_{w\in D_g}
\dist\bigl(g^i(w),g^i(x_g)\bigr).
\]
Letting $n\to\infty$, the left-hand side converges to $c_g>0$, whereas the right-hand side converges to $0$.
Hence
\[
c_g\le0,
\]
a contradiction.

Thus no nonempty subset $\Lambda'\subset M$ can satisfy the definition of $C^r$-strong pluripotency. This completes the proof.
\end{proof}

\section{Statistical pluripotency for Anosov diffeomorphisms}\label{s.pluripotency}

In this section, we prove
Theorem~\ref{thm:characterization-statistical-pluripotency} using Lemma~\ref{lem:SRB-basin}.
For the necessity, we use structural stability, uniqueness of weak-* limits, and the characterization
of the SRB measure as the unique equilibrium state
for the unstable Jacobian potential.
For the sufficiency, we simply take $g=f$.

\begin{proof}[Proof of Theorem~\ref{thm:characterization-statistical-pluripotency}]

Assume that $f$ is $C^r$-pluripotent for $\Lambda'$, and fix an arbitrary $x \in \Lambda'$. We show that $x \in B_f(\mu_f)$.

By the definition of pluripotency, there exist a sequence $g_k \to f$ in the $C^r$ topology and measurable sets $D_k \subset M$ of positive Riemannian volume such that
\begin{equation*}
    \lim_{n \to \infty}d_W \left(\delta_{y,g_k}^{n},\delta_{x_{g_k},g_k}^{n}\right)
    = 0
\end{equation*}
for every $y \in D_k$.
We may choose all $g_k$ in a sufficiently small neighborhood of $f$ so that each $g_k$ is a transitive Anosov diffeomorphism.

By structural stability, let $H_k:M\to M$ be the conjugacy defining the continuation $x_{g_k}=H_k(x)$.
Then
\[
g_k\circ H_k=H_k\circ f,
\qquad
H_k\to\mathrm{id}_M
\quad\text{uniformly}.
\]

By Lemma~\ref{lem:SRB-basin}, $g_k$ admits a unique SRB measure $\mu_{g_k}$ whose basin $B_{g_k}(\mu_{g_k})$ has full Riemannian volume.
Since $D_k$ has positive volume, we may choose
\begin{equation*}
  y_k \in D_k \cap B_{g_k}(\mu_{g_k}).
\end{equation*}

Then $\delta^n_{y_k,g_k}\to\mu_{g_k}$, and the pluripotency condition and the triangle inequality imply that
\begin{equation}\label{001}
  \lim_{n \to \infty}\delta^n_{x_{g_k},g_k}
  = \mu_{g_k}
\end{equation}
for each fixed $k$.
Since $H_k^{-1}$ is continuous, applying $(H_k^{-1})_*$ to \eqref{001} yields
\begin{equation}\label{002}
  \lim_{n \to \infty}\delta_{x,f}^n
  = (H_k^{-1})_*\mu_{g_k}.
\end{equation}

Since the sequence $\delta_{x,f}^n$ is independent of $k$ and the weak-* topology is Hausdorff, the limit in \eqref{002} is independent of $k$.
Thus there exists a probability measure $\nu$ such that
\[
(H_k^{-1})_*\mu_{g_k}=\nu
\qquad\text{for every }k.
\]
As $\nu$ is a weak-* limit of empirical measures of $f$, it is $f$-invariant.

It remains to identify $\nu$.

As shown in the proof of Lemma~\ref{lem:SRB-basin}, $\mu_{g_k}$ is the unique equilibrium state for $\phi_{g_k}^{(u)}$, and $P_{g_k}(\phi_{g_k}^{(u)})=0$.
Hence
\[
h_{\mu_{g_k}}(g_k)
+\int_M\phi_{g_k}^{(u)}\,d\mu_{g_k}=0.
\]
Since $(H_k)_*\nu=\mu_{g_k}$, invariance of measure-theoretic entropy under conjugacy and the pushforward integral formula give
\begin{equation}\label{003}
  h_\nu(f)
  +\int_M\phi_{g_k}^{(u)}\circ H_k\,d\nu=0.
\end{equation}

Since $g_k\to f$ in the $C^1$ topology, the unstable bundles converge uniformly, and hence $\phi_{g_k}^{(u)}\to\phi_f^{(u)}$ uniformly.
Together with $H_k\to\mathrm{id}_M$ uniformly and the uniform continuity of $\phi_f^{(u)}$, this yields
\[
\|\phi_{g_k}^{(u)}\circ H_k -\phi_f^{(u)}\|_\infty \leq
\|\phi_{g_k}^{(u)}-\phi_f^{(u)}\|_\infty +\|\phi_f^{(u)}\circ H_k-\phi_f^{(u)}\|_\infty
 \to 0.
\]
Letting $k\to\infty$ in \eqref{003}, we obtain
\[
h_\nu(f)+\int_M\phi_f^{(u)}\,d\nu
=0=P_f(\phi_f^{(u)}).
\]
Thus $\nu$ is an equilibrium state for $\phi_f^{(u)}$, and hence $\nu=\mu_f$ by uniqueness.
By \eqref{002} and $\nu=\mu_f$, we have $\delta_{x,f}^n\to\mu_f$, so $x\in B_f(\mu_f)$.
Since $x\in\Lambda'$ was arbitrary, we conclude that $\Lambda'\subset B_f(\mu_f)$.

We now prove the converse.
By Lemma~\ref{lem:SRB-basin}, $f$ admits a unique SRB measure $\mu_f$ whose basin $B_f(\mu_f)$ has full Riemannian volume.
In particular, $B_f(\mu_f)$ is a nonempty measurable set.

For every $z \in M$, we have
\begin{equation*}
  d_W(\delta^n_{f(z),f},\delta^n_{z,f})
  \leq
  \frac{\dist(f^n(z),z)}{n}
  \leq
  \frac{\operatorname{diam}(M)}{n}.
\end{equation*}
It follows that $z \in B_f(\mu_f)$ if and only if $f(z) \in B_f(\mu_f)$.
Thus $B_f(\mu_f)$ is $f$-invariant.

Suppose that $\Lambda'\subset B_f(\mu_f)$.
Fix $x\in\Lambda'$ and a $C^r$-neighborhood $\mathcal U$ of $f$.
Take $g=f$ and $D_g=B_f(\mu_f)$.
By Lemma~\ref{lem:SRB-basin}, $D_g$ is measurable and has full Riemannian volume.
Since $g=f$, we have $x_g=x$.
For $y \in D_g$, both $x$ and $y$ belong to $B_f(\mu_f)$.
By the triangle inequality, we obtain
\begin{equation*}
  \lim_{n \to \infty}d_W(\delta^n_{y,g}, \delta^n_{x_g,g})
  = 0.
\end{equation*}
Thus $f$ is $C^r$-pluripotent for $\Lambda'$.
\end{proof}

\begin{proof}[Proof of Corollary~\ref{cor:absence-pluripotency}]
  Let $\mu_f$ be the unique SRB measure of $f$, and let $p$ be a periodic point of $f$ with period $\pi(p)$.
  Define
  \begin{equation*}
    \nu_p
    := \frac{1}{\pi(p)} \sum_{i=0}^{\pi(p)-1}\delta_{f^i(p)}.
  \end{equation*}
  Then
  \begin{equation*}
    \lim_{n \to \infty}\delta^n_{p,f}
    = \nu_p
  \end{equation*}
  in the weak-* topology.
  Since
  \begin{equation*}
    \operatorname{supp}\nu_p=\operatorname{Orb}_f(p)
  \end{equation*}
  is finite, whereas
  \begin{equation*}
    \operatorname{supp}\mu_f=M
  \end{equation*}
  we have $\nu_p\ne\mu_f$.
  If $p\in B_f(\mu_f)$, then $\delta^n_{p,f}\to\mu_f$ by the definition of the basin.
  The uniqueness of the weak-* limit would imply $\mu_f=\nu_p$, a contradiction.
  Thus $p\notin B_f(\mu_f)$.

  Since $f$ has periodic points, $B_f(\mu_f)$ is a proper subset of $M$.
  Applying Theorem~\ref{thm:characterization-statistical-pluripotency} with $\Lambda'=M$ yields the conclusion.
\end{proof}

\section*{Acknowledgments}
We thank Yushi Nakano for discussions on whether Anosov diffeomorphisms
can exhibit pluripotency. We also thank Masayuki Asaoka for discussions
related to this work.

\end{document}